\documentclass[a4paper,11pt,oneside]{amsart}
\usepackage{graphics}
\usepackage{amsmath}
\usepackage{latexsym,lscape}
\usepackage{amsthm,float,autograph}
\usepackage{epic,latexsym,bezier,caption,amsbsy,psfrag,color,enumerate,amsfonts,amsmath,amscd,amssymb,comment}
\usepackage{epsfig}
\usepackage{rotating}
\usepackage{graphics}
\usepackage[latin1]{inputenc}
\usepackage{multicol}
\usepackage{epstopdf}

\newtheorem{defi}{Definition}[section]
\newtheorem{theorem}[defi]{Theorem}
\newtheorem{lemma}[defi]{Lemma}
\newtheorem{prop}[defi]{Proposition}

\theoremstyle{definition}

\newtheorem{example}[defi]{Example}

\newtheorem{remark}[defi]{Remark}

\begin{document}
\keywords{Switching, cospectral, isospectral, adjacency spectrum}

\title{Constructing cospectral signed graphs}

\author[F. Belardo]{Francesco Belardo}
\address{Francesco Belardo, Universit\`{a} degli Studi di Napoli Federico II, Via Cintia, 80126 Napoli, Italia}
\email{fbelardo@unina.it}

\author[M. Brunetti]{Maurizio Brunetti}
\address{Maurizio Brunetti, Universit\`{a} degli Studi di Napoli Federico II, Via Cintia, 80126 Napoli, Italia}
\email{mbrunett@unina.it}

\author[M. Cavaleri]{Matteo Cavaleri}
\address{Matteo Cavaleri, Universit\`{a} degli Studi Niccol\`{o} Cusano - Via Don Carlo Gnocchi, 3 00166 Roma, Italia}
\email{matteo.cavaleri@unicusano.it}

\author[A. Donno]{Alfredo Donno}
\address{Alfredo Donno, Universit\`{a} degli Studi Niccol\`{o} Cusano - Via Don Carlo Gnocchi, 3 00166 Roma, Italia}
\email{alfredo.donno@unicusano.it}

\begin{abstract}
A well--known fact in Spectral Graph Theory is the existence of pairs of isospectral nonisomorphic graphs (known as PINGS). The work of A.J. Schwenk (in 1973) and of C. Godsil and B. McKay (in 1982) shed some light on the explanation of the presence of isospectral graphs, and they gave routines to construct PINGS. Here, we consider the Godsil-McKay--type routines developed for graphs, whose adjacency matrices are $(0,1)$-matrices, to the level of signed graphs, whose adjacency matrices allow the presence of $-1$'s. We show that, with suitable adaption, such routines can be successfully ported to signed graphs, and we can built pairs of cospectral switching nonisomorphic signed graphs.
\end{abstract}

\maketitle

\begin{center}
{\footnotesize{\bf Mathematics Subject Classification (2010)}: 05C22, 05C50}
\end{center}

\section{Introduction}

In this manuscript we consider signed graphs, that are (simple and finite) graphs whose egdes get a sign $\pm 1$. More formally, a signed graph $\Gamma=(G,\sigma)$ is a graph $G=(V_G,E_G)$, with vertex set $V_G$ and edge set $E_G$, together with a function $\sigma:E_G\rightarrow\{+1,-1\}$ assigning a positive or negative sign to each edge. The graph $G$ is the underlying graph of $\Gamma$, while $\sigma$ is the signature of $\Gamma$. In this way, the adjacency matrix $A(\Gamma)=A_\Gamma$ of $\Gamma$ is defined similarly to (unsigned) graphs, that is, by putting $+1$ or $-1$ in correspondence to positive or negative edges, respectively. A walk is positive or negative if the product of corresponding weights is positive or negative, respectively. Since cycles are special kinds of walks, this definition applies to them as well, and we get positive and negative cycles.

Those notions related to unsigned graphs directly extend to signed graphs. For example, the degree $d(v)$ of a vertex $v$ in $\Gamma$ is simply its degree in $G$. A vertex of degree one is said to be a pendant vertex. However, some other definitions comes from signed graph theory and they depend on the signature, as, for example, the positive $d^+(v)$ (resp.,  negative $d^-(v)$) degree of vertex $v$ is the number of positive (negative) edges incident to $v$, or the already mentioned sign of a cycle. A signed graph is {balanced} if all its cycles are positive, otherwise it is {unbalanced}. Unsigned graphs can be included in this theory as signed graphs where all edges get a positive sign, that is, the {all-positive signature}; clearly, unsigned graphs are balanced graphs (and the vice versa is also true).

An unavoidable feature of signed graphs is the concept of {signature switching}. Given a signed graph $\Gamma=(G,\sigma)$ and a subset $U\subseteq V_G$, the signed graph $\Gamma^U$, obtained from $\Gamma$ by reversing the edge signs in the cut $[U,V_G\setminus U]$, is said to be (switching) equivalent to $\Gamma$, and $\sigma_{\Gamma^U}$ to $\sigma_\Gamma$. We write $\Gamma^U \sim \Gamma$ or $\sigma_{\Gamma^U} \sim \sigma_{\Gamma}$. Notably, the signature switching does not affect the sign of cycles, hence $\Gamma$ and $\Gamma^U$ share the same positive (and negative) cycles. Evidently, the signature is determined up to equivalence by the set of positive cycles (see \cite{zas2}). Signatures equivalent to the all-positive one lead to balanced signed graphs (that are equivalent to unsigned graphs).

As (unsigned) graphs under vertex permutations are (naturally) considered as the same graph, we can combine switching equivalence and vertex permutation to the more general concept of switching isomorphism of signed graphs. Here, switching isomorphic signed graphs are considered as the same signed graph. For basic results in the theory of signed graphs, the reader is referred to Zaslavsky \cite{zas2} (see also the dynamic survey \cite{dynamic1}). \smallskip

We next consider the (adjacency) spectral theory of signed graphs. Recall, the adjacency matrix $A_{\Gamma}=(a_{ij})$ is the symmetric $(0,+1,-1)$-matrix such that $a_{ij}=\sigma(ij)$ whenever the vertices $i$ and $j$ are adjacent, and $a_{ij}=0$ otherwise. For a signed graph $\Gamma=(G,\sigma)$ and its adjacency matrix $A_\Gamma$, the $A$-polynomial is $p_{\Gamma}(x)=\det(xI-A_\Gamma)$. The spectrum of $A_\Gamma$ is called the $A$-spectrum of the signed graph $\Gamma$. We shall omit the suffices or indices if clear from the context, and similar notation will be used for unsigned graphs.

Switching has a matrix counterpart. In fact, let $\Gamma$ and $\Gamma^U$ be two switching equivalent graphs. Consider the matrix $S_U=\textrm{diag}(s_1,s_2,\ldots,s_n)$ such that
\[s_i=\left\{
\begin{array}{ll}
+1,& i\in U;\\
-1, & i\in V_G\setminus U.
\end{array}
\right.\]
The matrix $S_U$ is the {\it switching matrix}. It is easy to check that
\[A_{\Gamma^U}=S_U\, A_{\Gamma}\, S_U.\]
Hence, signed graphs from the same switching class share similar graph matrices by means of signature matrices (signature similarity). If we also allow permutation of vertices, we have signed permutation matrices, and we can speak of (switching) isomorphic signed graphs. Switching isomorphic signed graphs are cospectral, and their matrices are signed-permutationally similar. For basic results on the spectra of signed graphs (resp. unsigned graphs), we refer the readers to \cite{BeCiKoWa,zas1} (resp. \cite{BrHae,CvDS,CvRoSi}).

It is well--known that in general graphs are not determined by their set of eigenvalues with respect to some prescribed graph matrix. In fact, since the early papers on Spectral Graph Theory, it was evident that there are infinite families of nonisomorphic graphs sharing the same spectrum. Such occurrences prompted two lines of research: 1) find ways to built (eventually nonisomorphic) cospectral graphs, and 2) prove whether some class of graphs are or are not determined by their eigenvalues (by possibly detecting all exceptions). It seems to us that the second research line has much more literature, especially after the seminal paper \cite{HaeSpe}. However, here we are more interested in the first line of research, possibly started by the nice paper of Schwenk \cite{Schwenk} and its unofficial sequel by Godsil and McKay \cite{godsilmckay}. In fact, the latter paper presented
routines for (possibly) construct pairs of non-isomorphic cospectral graphs (know as PINGS in \cite{CvDS}), and such method is nowadays known as ``Godsil-McKay'' switching. A curiosity is that, the first concept of switching is due to Seidel: the Seidel switching lead to isospectral graphs with respect to the Seidel matrix (cf. \cite{BrHae}, for example). As the Seidel matrix of a (simple and unsigned) graph can be seen as the adjacency matrix of a signed complete graph, then Zaslavsky elaborated the concept of signature switching. On the other hand, the Seidel switching also influenced the Godsil-McKay switching. Hence, all these (apparently) different concepts of switching share the same root.

The aim of this paper is to extend to the spectral theory of signed graphs the Godsil-McKay switching in order to built cospectral signed graphs.\\ We describe the remainder of this paper. In Section \ref{sectionGMunsigned} we recall the basic idea of the Godsil-McKay switching, which will be revised in Section \ref{sect3} to work with the adjacency matrix of signed graphs. In Section \ref{sect4} we consider the ``generalized'' Godsil-McKay switching, recently developed for unsigned graphs in \cite{wangqiuhu}, and in Section \ref{sect5} we adapt it to signed graphs as well.

\section{The GM-switching for unsigned graphs}\label{sectionGMunsigned}
In this section we restrict to unsigned graphs. For $G=(V_G,E_G)$, simple and connected graph, we denote by $A_G$ the corresponding adjacency matrix, and by $p_G(x)$ the characteristic polynomial of $G$. Recall, two graphs are said to be cospectral if they share the same characteristic polynomial.\\ In \cite{godsilmckay}, Godsil and McKay introduced a graph construction, which is nowadays called the  Godsil-McKay switching, or GM-switching for short, allowing to produce pairs of cospectral graphs. To keep the paper self-contained, we will briefly explain the construction used in the GM-switching (for some recent developments see \cite{AbBuHa}).

Let $\pi=\{C_1, \ldots, C_t, D\}$ be a partition of $V_G$, with $|C_i|=n_i$ and $|D|=d$. Suppose that, for each $1\leq i,j \leq t$ and $v\in D$:
\begin{enumerate}
\item any two vertices in $C_i$ have the same number of neighbors in $C_j$;
\item $v$ has either $0$, $n_i/2$ or $n_i$ neighbors in $C_i$.
\end{enumerate}
Then the graph $G^{\pi}$ constructed from $G$ by local switching with respect to the partition $\pi$ is obtained as follows: for every $v\in D$ and $1\leq i\leq t$ such that $v$ has $n_i/2$ neighbors in $C_i$, one deletes these $n_i/2$ edges and join $v$ instead to the remaining $n_i/2$ vertices of $C_i$. It turns out that the graphs $G$ and $G^\pi$ are cospectral.\\
\indent Before giving a sketch of the proof from \cite{godsilmckay}, we need to fix some notations. We will denote by $O_{m,n}$ (resp. $J_{m,n}$) the $m\times n$ matrix whose entries are all equal to $0$ (resp. to $1$), and by $I_n$ the identity matrix of order $n$. We also put $O_n=O_{n,n}$ and $J_n = J_{n,n}$. We will often write ${\bf 0}_n$ and ${\bf 1}_n$ instead of $O_{n,1}$ and $J_{n,1}$, respectively. Finally, given $n$ square matrices $A_1, \ldots, A_n$, we denote by ${\rm diag}(A_1, \ldots, A_n)$ the matrix block-diagonal matrix, whose diagonal blocks are whose $i$-th diagonal block is $A_i$, and all the remaining blocks are zero-matrices.

\begin{theorem}\cite{godsilmckay}\label{godsilmckaythm}
$G$ and $G^\pi$ are cospectral.
\end{theorem}
\begin{proof}
For each positive integer $m$, put
\begin{eqnarray}\label{defiQm}
Q_m = \frac{2}{m}J_m-I_m.
\end{eqnarray}
The following properties, called the {\bf $Q$-properties}, can be easily verified:
\begin{itemize}
\item $Q_m^2=I_m$;
\item if $X$ is an $m\times n$ matrix with constant row sums and constant column sums, then $Q_mXQ_n = X$;
\item if $\textbf{x}$ is a vector with $m$ entries, $m/2$ of which are $0$ and $m/2$ of which are $1$, then $Q_m\textbf{x} = {\bf 1}_m-\textbf{x}$.
\end{itemize}
By choosing a suitable ordering in $V_G$, we get that the adjacency matrix of $G$ is
$$
A_G = \left(
      \begin{array}{ccccc}
        C_1 & C_{12} & \cdots & C_{1t} & D_1 \\
        C_{12}^T & C_2 & \cdots & C_{2t} & D_2 \\
        \vdots & \vdots &  & \vdots & \vdots \\
        C_{1t}^T & C_{2t}^T & \cdots & C_t & D_t \\
        D_1^T & D_2^T & \cdots & D_t^T & D \\
      \end{array}
    \right)
$$
where:
\begin{itemize}
\item for each $1\leq i,j\leq t$, the matrix $C_i$ is symmetric of order $n_i$, and the matrix $C_{ij}$ is $n_i\times n_j$;
\item for each $i=1,\ldots, t$, the matrix $D_{i}$ is $n_i\times d$;
\item $D$ is a symmetric matrix of order $d$.
\end{itemize}
Moreover, for each $i=1,\ldots, t$, let us denote by ${\bf d}_i^j$ the $j$-th column of $D_i$, with $j=1,\ldots, d$. The assumptions on the partition $\pi$ ensure that each $C_i$ and $C_{ij}$ has constant row sums and constant column sums, and each column ${\bf d}_i^j$ has either $0$, $n_i/2$ or $n_i$ entries equal to $1$. If we put $Q = {\rm diag} (Q_{n_1}, Q_{n_2},\ldots, Q_{n_t}, I_{d})$, with $Q_{n_i}$ as in Eq. \eqref{defiQm}, then
$$
QA_GQ  =  \left(
      \begin{array}{ccccc}
        C_1 & C_{12} & \cdots & C_{1t} & \tilde{D}_1 \\
        C_{12}^T & C_2 & \cdots & C_{2t} & \tilde{D}_2 \\
        \vdots & \vdots &  & \vdots & \vdots \\
        C_{1t}^T & C_{2t}^T & \cdots & C_t & \tilde{D}_t \\
        \tilde{D}_1^T & \tilde{D}_2^T & \cdots & \tilde{D}_t^T & D \\
      \end{array}
    \right)
$$
where, for each $i=1,\ldots, t$, and for each column $\tilde{{\bf d}}_{i}^j$ of $\tilde{D}_i$, one has
$$
\tilde{{\bf d}}_{i}^j=Q_{n_i}{\bf d}_{i}^j= \left\{
                                              \begin{array}{ll}
                                               {\bf d}_{i}^j & \hbox{if } {\bf d}_{i}^j= {\bf 0}_{n_i} \hbox{ or } {\bf d}_{i}^j= {\bf 1}_{n_i}\\
                                                {\bf 1}_{n_i}-{\bf d}_{i}^j & \hbox{otherwise.}
                                              \end{array}
                                            \right.
$$
Therefore, $QA_GQ$ is the adjacency matrix of $G^\pi$. As $Q^2=I_{|V_G|}$, the matrices $QA_GQ$ and $A_G$ are similar, and the claim is proved.
\end{proof}

A variant of the previous theorem can be adapted to $(0,1)$-matrices as follows \cite{HaeSpe}.
\begin{prop}\label{propgodsilunsigned}
Let $N$ be a $(0,1)$-matrix of order $b\times c$ whose column sums
are $0, b$ or $b/2$. Define $\tilde{N}$ to be the matrix obtained from $N$ by
replacing each column ${\bf n}$ with $b/2$ ones by its complement ${\bf 1}_b-{\bf n}$.
Let $B$ be a symmetric $b\times b$ matrix with constant row (and
column) sums, and let $C$ be a symmetric $c\times c$ matrix. Then the matrices
$$
M= \left(
     \begin{array}{cc}
       B & N \\
       N^T & C \\
     \end{array}
   \right)    \qquad \mbox{and } \qquad \tilde{M} = \left(
                                         \begin{array}{cc}
                                           B & \tilde{N} \\
                                           \tilde{N}^T & C \\
                                         \end{array}
                                       \right)
$$
are cospectral.
\end{prop}

\begin{example}\label{exexex}\rm
The graph $G$ on the left of Fig. \ref{examplegodsilnonsigned} satisfies the hypotheses of Theorem \ref{godsilmckaythm}, with $V_G = C_1 \sqcup D$, where
$$
C_1=\{1,2,3,4,5,6,7,8\} \qquad D=\{9,10\}.
$$
On the right, the graph $G^\pi$ obtained by GM-switching. Notice that $G$ and $G^\pi$ are cospectral by Theorem \ref{godsilmckaythm}, but they are clearly non-isomorphic. Finally, we have
$$
p_G(x) = p_{G^{\pi}}(x) = x(x+2)(x+1)(x^7-3x^6-9x^5+23x^4+23x^3-43x^2-18x+16).
$$
\begin{figure}[h]
\begin{center}
\psfrag{1}{$1$}\psfrag{2}{$2$} \psfrag{3}{$3$} \psfrag{4}{$4$} \psfrag{5}{$5$} \psfrag{6}{$6$} \psfrag{7}{$7$} \psfrag{8}{$8$} \psfrag{9}{$9$} \psfrag{10}{$10$}
\psfrag{G}{$G$} \psfrag{Gpi}{$G^\pi$}
\includegraphics[width=0.5\textwidth]{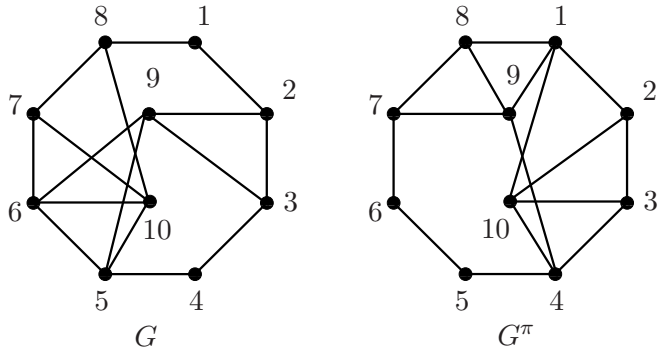} \caption{The graphs $G$ and $G^\pi$ of Example \ref{exexex}.}  \label{examplegodsilnonsigned}
\end{center}
\end{figure}
\end{example}

\begin{remark}\rm
The idea of the GM-switching is related to the notion of ``equitable partition'' of the vertex set of a graph.
Equitable partitions have a number of significant applications in Graph theory: for example, the vertex set partition of a graph under the action of a group of automorphisms is always equitable. This fact has been used in the context of graph isomorphism algorithms (we refer to the book \cite{godsil} for more details).\\
\indent Let $u\in V_G$ and denote by $N(u)$ the neighborhood of $u$ in $G=(V_G,E_G)$, that is, $N(u)=\{v\in V_G: v\sim u\}$. Given a partition $\pi=\{C_1,\ldots, C_t\}$ of $V_G$, one says that $\pi$ is equitable if, for all $i\in \{1,\ldots, t\}$, and all $v,w\in C_i$, one has $|N(v)\cap C_j|=|N(w)\cap C_j|$, for each $j\in \{1,\ldots, t\}$ (see, for instance, \cite{godsil2,godsil80} for several equivalent characterizations of equitability).
It is then clear that the conditions required in the GM-switching are equivalent to say that $\{C_1,C_2,\ldots, C_t\}$ is an equitable partition of the subgraph of $G$ induced by the vertices in $V_G\setminus D$.
\end{remark}

\section{A GM-switching for signed graphs}\label{sect3}
In this section we are going to define a notion of GM-switching for signed graphs. In other words, we want to extend the results of the previous section from the setting of $(0,1)$-matrices to the set of $(-1,0,1)$-matrices. We already know that, for every positive integer $m$, the matrix $Q_m = \frac{2}{m}J_m-I_m$ satisfies the $Q$-properties given in Section \ref{sectionGMunsigned}. We also need the following lemma.

\begin{lemma}\label{sixconditions}
Let $Q_m = \frac{2}{m}J_m-I_m$, and ${\bf x}=(x_i)_{i=1,\ldots, m}$ be a vector with entries in $\{-1,0,1\}$. Then:
\begin{enumerate}
\item if ${\bf x}$ is a vector whose entries have sum $0$, then $Q_{m}{\bf x} = -{\bf x}$;
\item if $m$ is even, and ${\bf x}$ is a vector with half entries equal to $0$, and half entries equal to $1$, then $Q_m{\bf x} = {\bf 1}_m-{\bf x}$;
\item if $m$ is even, and ${\bf x}$ is a vector with half entries equal to $0$, and half entries equal to $-1$, then $Q_m{\bf x}= -{\bf 1}_m-{\bf x}$;
\item if ${\bf x}= {\bf 1}_m$, then $Q_{m}{\bf x} = {\bf x}$;
\item if ${\bf x}=-{\bf 1}_m$, then $Q_{m}{\bf x} = {\bf x}$.
\end{enumerate}
\end{lemma}
\begin{proof}
We only prove (1) and (4), because (2) follows from the $Q$-properties, (3) follows from (2), and (5) follows from (4). First of all, notice that
$$
(Q_m)_{ij}=q_{ij}= \left\{
                                                     \begin{array}{ll}
                                                       \frac{2}{m} & \hbox{if } i\neq j\\
                                                       \frac{2-m}{m} & \hbox{if } i=j.
                                                     \end{array}
                                                   \right.
$$
For (1), we have
$$
(Q_m{\bf x})_{i} = \sum_{k=1}^m q_{ik}x_{k} = \frac{2}{m}(0-x_{i})+\frac{2-m}{m}x_i= -x_i
$$
and the claim is proved. The property (4) follows from the fact that the matrix $Q_m$ has constant row and column sums equal to $1$, by definition.
\end{proof}

Now let $\Gamma=(G,\sigma)$ be a signed graph with underlying graph $G$, vertex set $V_G$ and edge set $E_G$, with signature function $\sigma:E_G\to \{\pm 1\}$. The graph $G$ is supposed to be simple. For a given vertex $v\in V_G$, let us denote by $d^+(v)$ the number of positive edges incident to $v$ in $\Gamma$, and by  $d^-(v)$ the number of negative edges incident to $v$ in $\Gamma$. Moreover, we put $d^{\pm}(v) = d^+(v) -d^-(v)$, which is usually called the \emph{net-degree} of $v$ in $\Gamma$.

Now assume that $\pi=\{C_1, \ldots, C_t, D\}$ is a partition of $V_G$, with $|C_i|=n_i$ for each $i=1,\ldots, t$, and $|D|=d$. For a given vertex $v\in V_G$, let $d^+_i(v)$ denote the number of positive edges connecting $v$ to some vertex of $C_i$, and similarly for $d^-_i(v)$, so that the $i$-th net-degree of $v$ is defined as $d^{\pm}_i(v) = d^+_i(v)-d^-_i(v)$. Now, suppose that:
\begin{enumerate}
\item for each $1\leq i,j \leq t$, any two vertices in $C_i$ have the same $j$-th net-degree;
\item for each $i=1,\ldots, t$ and $v\in D$:
\begin{itemize}
\item either $d^{\pm}_i(v)=0$ (in particular, $v$ may have no neighbor in $C_i$);
\item or $d^+_i(v) = n_i/2$ and $d^-_i(v)=0$;
\item or $d^-_i(v) = n_i/2$ and $d^+_i(v)=0$;
\item or $d^+_i(v)=n_i$;
\item or $d^-_i(v)=n_i$.
\end{itemize}
\end{enumerate}
Then the graph $\Gamma^{\pi}$ constructed from $\Gamma$ by local switching with respect to the partition $\pi$ is obtained as follows: for every $v\in D$ and $1\leq i\leq t$:
\begin{itemize}
\item if $d^{\pm}_i(v)=0$, then one changes the sign of any edge connecting $v$ to a vertex of $C_i$;
\item if $d^+_i(v) = n_i/2$ and $d^-_i(v)=0$, then one deletes the existing (positive) edges from $v$ to $C_i$ and join instead $v$ to the remaining $n_i/2$ vertices of $C_i$ by positive edges;
\item if $d^-_i(v) = n_i/2$ and $d^+_i(v)=0$, then one deletes the existing (negative) edges from $v$ to $C_i$ and join instead $v$ to the remaining $n_i/2$ vertices of $C_i$ by negative edges.
\end{itemize}

\begin{theorem}\label{godsilmckaysignedthm}
 The signed graphs $\Gamma$ and $\Gamma^\pi$ are cospectral.
\end{theorem}
\begin{proof}
By choosing a suitable ordering in $V_G$, the signed adjacency matrix of $\Gamma$ can be written as
$$
A_\Gamma = \left(
      \begin{array}{ccccc}
        C_1 & C_{12} & \cdots & C_{1t} & D_1 \\
        C_{12}^T & C_2 & \cdots & C_{2t} & D_2 \\
        \vdots & \vdots &  & \vdots & \vdots \\
        C_{1t}^T & C_{2t}^T & \cdots & C_t & D_t \\
        D_1^T & D_2^T & \cdots & D_t^T & D \\
      \end{array}
    \right)
$$
where:
\begin{itemize}
\item for each $1\leq i\leq t$, the matrix $C_i$ is symmetric of order $n_i$, with constant row (and column) sums;
\item for each $1\leq i,j\leq t$, the matrix $C_{ij}$ is an $n_i\times n_j$ matrix with constant row sums and column sums;
\item for each $1\leq i\leq t$, every column vector of $D_i$ satisfies one of the five conditions of Lemma \ref{sixconditions};
\item the matrix $D$ is symmetric of order $d$.
\end{itemize}
If we put $Q = {\rm diag} (Q_{n_1}, Q_{n_2},\ldots, Q_{n_t}, I_{d})$, where $Q_{n_i}$ is defined as in Eq. \eqref{defiQm}, then
$$
QA_\Gamma Q  =  \left(
      \begin{array}{ccccc}
        C_1 & C_{12} & \cdots & C_{1t} & \tilde{D}_1 \\
        C_{12}^T & C_2 & \cdots & C_{2t} & \tilde{D}_2 \\
        \vdots & \vdots &  & \vdots & \vdots \\
        C_{1t}^T & C_{2t}^T & \cdots & C_t & \tilde{D}_t \\
        \tilde{D}_1^T & \tilde{D}_2^T & \cdots & \tilde{D}_t^T & D \\
      \end{array}
    \right)
$$
where, for each $i=1,\ldots, t$, the column vectors of $D_i$ change in $\tilde{D}_i$ according to Lemma \ref{sixconditions}. Therefore, $QA_\Gamma Q$ is the adjacency matrix of $\Gamma^\pi$, and the statement follows.
\end{proof}

\begin{example}\label{example2}\rm
Let $\Gamma=(G,\sigma)$ be the signed graph with adjacency matrix

$$
A_\Gamma = \left(
      \begin{array}{ccc}
        C_1 & C_{12} & D_1 \\
C_{12}^T & C_2 & D_2 \\
 D_1^T & D_2^T & D \\ \end{array}
    \right)=\left(
   \begin{array}{ccc|cccc|c}
     0 & 1 & 1 & 0 & 0 & -1 & 1 & 0 \\
     1 & 0 & 1 & -1 & 0 & 1 & 0 & 1 \\
     1 & 1 & 0 & 1 & 0 & 0 & -1 & -1 \\  \hline
     0 & -1 & 1 & 0 & -1 & 0 & 1 & 1 \\
     0 & 0 & 0 & -1 & 0 & 1 & 0 & 1 \\
     -1 & 1 & 0 & 0 & 1 & 0 & -1 & 0 \\
     1 & 0 & -1 & 1 & 0 & -1 & 0 & 0 \\   \hline
     0 & 1 & -1 & 1 & 1 & 0 & 0 & 0 \\
   \end{array}
 \right)
$$

which is depicted on the left in Fig. \ref{examplegodsilsigned}. In particular, $C_1=\{1,2,3\}$, $C_2=\{4,5,6,7\}$, and $D=\{8\}$.
\begin{figure}[h]
\begin{center}
\psfrag{1}{$1$}\psfrag{2}{$2$} \psfrag{3}{$3$} \psfrag{4}{$4$} \psfrag{5}{$5$} \psfrag{6}{$6$} \psfrag{7}{$7$} \psfrag{8}{$8$}
\psfrag{Gamma}{$\Gamma$} \psfrag{Gammapi}{$\Gamma^\pi$}
\includegraphics[width=0.8\textwidth]{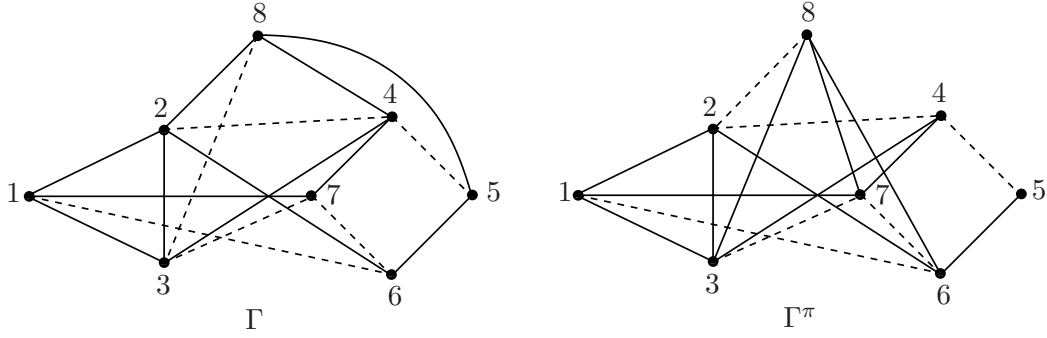} \caption{The graphs $\Gamma$ and $\Gamma^\pi$ of Example \ref{example2}.}  \label{examplegodsilsigned}
\end{center}
\end{figure}
By putting $Q_3 = \frac{2}{3}J_3-I_3$ and $Q_4 = \frac{1}{2}J_4-I_4$, and $Q = {\rm diag} (Q_{3}, Q_{4}, I_{1})$, and computing $QA_{\Gamma}Q$, one obtains the adjacency matrix of the graph $\Gamma^\pi$, represented on the right in Fig. \ref{examplegodsilsigned}. Therefore, the graphs $\Gamma$ and $\Gamma^\pi$ are cospectral, and
$$
p_\Gamma(x) = p_{\Gamma^{\pi}}(x) = (x-2)(x^7+2x^6-13x^5-14x^4+39x^3+18x^2-20x-8).
$$
On the other hand, $\Gamma$ and $\Gamma^{\pi}$ are not switching isomorphic, as the vertex $5$ has degree two in $\Gamma^\pi$, but no vertex has degree two in $\Gamma$.
\end{example}

As a consequence of Theorem \ref{godsilmckaysignedthm}, we get the following proposition, which is a $(-1,0,1)$-matrix analogue of Proposition \ref{propgodsilunsigned}.
\begin{prop}\label{propgodsilsigned}
Let $N$ be a $(-1,0,1)$-matrix of order $b\times c$ whose column ${\bf n}^j$ may be such that:
\begin{itemize}
\item either $\sum_{h=1}^bn^j_h=0$;
\item or $b/2$ entries of ${\bf n}^j$ are equal to $1$ and $b/2$ entries of ${\bf n}^j$ are equal to $0$;
\item or $b/2$ entries of ${\bf n}^j$ are equal to $-1$ and $b/2$ entries of ${\bf n}^j$ are equal to $0$;
\item or ${\bf n}^j={\bf 1}_b$;
\item or ${\bf n}^j=-{\bf 1}_b$.
\end{itemize}
Define $\tilde{N}$ to be the matrix obtained from $N$ by
replacing each column ${\bf n}^j$ whose entries have sum $0$ with $-{\bf n}^j$; each column ${\bf n}^j$ with half entries equal to $1$ and half entries equal to $0$ with ${\bf 1}_b-{\bf n}^j$; and finally each column ${\bf n}^j$ with half entries equal to $-1$ and half entries equal to $0$ with $-{\bf 1}_b-{\bf n}^j$. Let $B$ be a symmetric $b\times b$ matrix with constant row (and column) sums, and let $C$ be a symmetric $c\times c$ matrix. Then the matrices
$$
M= \left(
     \begin{array}{cc}
       B & N \\
       N^T & C \\
     \end{array}
   \right)    \qquad \mbox{and } \qquad \tilde{M} = \left(
                                         \begin{array}{cc}
                                           B & \tilde{N} \\
                                           \tilde{N}^T & C \\
                                         \end{array}
                                       \right)
$$
are cospectral.
\end{prop}

\begin{remark}\rm
Observe that our Theorem \ref{godsilmckaysignedthm} contains and extends the classical construction by Godsil and McKay (Theorem \ref{godsilmckaythm}), as in the unsigned case the signature is constantly equal to $1$, and the condition $d^{\pm}_i(v)=0$ is equivalent to the non existence of edges connecting the vertex $v$ to the part $C_i$.
\end{remark}

\section{A different switching for unsigned graphs}\label{sect4}
The GM-switching described in Section \ref{sectionGMunsigned} is obtained by conjugating the adjacency matrix of the graph $G=(V_G,E_G)$ by the rational orthogonal matrix $Q$. In the paper \cite{wangqiuhu} the authors define a new construction, based on a different rational matrix, and allowing to produce pairs of non-isomorphic cospectral graphs which cannot be obtained by the classical GM-switching. We will refer to this construction as the ``Generalized GM-switching'' (or, G-GM-switching for short), and we recall it in what follows.

Let $G=(V_G,E_G)$ be a graph, and let $p$ be an odd prime. Suppose that the vertex set $V_G$ admits the partition $V_1\sqcup V_2 \sqcup V$, with $|V_1|=|V_2|=p$ and $|V|=d$.
For $i=1,2$, let $d_i(v)$ be the number of edges connecting $v$ to some vertex of $V_i$, and assume that:
\begin{itemize}
\item $d_1(v)-d_2(v) = d_2(u)-d_1(u) = \ell$ for all $v\in V_1$ and $u\in V_2$;
\item every vertex in $V$ either is adjacent to all the vertices in $V_1$ and has no neighbor in $V_2$, or is adjacent to all the vertices in $V_2$ and has no neighbor in $V_1$, or has the same number of neighbors in $V_1$ and $V_2$.
\end{itemize}
For such a graph $G$, the G-GM-switching producing the new graph $G'$ works as follows: for every vertex $v\in V$ which is adjacent to exactly all the vertices of $V_1$ (resp. $V_2$), switch the edges so that $v$ is adjacent to exactly all the vertices of $V_2$ (resp. $V_1$) in $G'$, and leave all the other edges unchanged.

\begin{theorem}\cite{wangqiuhu}\label{wangqiuhu}
The graphs $G$ and $G'$ are cospectral.
\end{theorem}
The proof is based on the fact that the adjacency matrices $A_G$ and $A_{G'}$ are conjugated by the rational orthogonal matrix $Q = {\rm diag}(U_{2p}, I_d)$, where $U_{2p} = I_{2p}+\frac{1}{p}\left(
                                                                                                                                                                  \begin{array}{cc}
                                                                                                                                                                    -J_p & J_p \\
                                                                                                                                                                    J_p & -J_p \\
                                                                                                                                                                  \end{array}
                                                                                                                                                                \right)$
and
$$
A_G = \left(
           \begin{array}{cc}
             M & C \\
             C^T & D \\
           \end{array}
         \right),  \qquad \textrm{with } \qquad  M=\left(
           \begin{array}{cc}
             A_1 & B \\
             B^T & A_2 \\
           \end{array}
         \right).
$$
Observe that the hypothesis on the graph $G$ implies that:
\begin{enumerate}
\item $A_1$ and $A_2$ are symmetric matrices of order $p$;
\item for each $i,i'=1,\ldots ,p$, one has $\sum_{j=1}^p ((A_1)_{ij}-B_{ij}) =  \sum_{j=1}^p ((A_2)_{i'j}-B_{ji'}) =\ell$;
\item each column ${\bf c}^j$ of $C$ is a vector of length $2p$ which can be partitioned into two semicolumns ${\bf c}^{j,1}$ and ${\bf c}^{j,2}$ of length $p$ as ${\bf c}^j = \left(
                                                                                                                                    \begin{array}{c}
                                                                                                                                      c^{j,1}_1 \\
                                                                                                                                      \vdots \\
                                                                                                                                      c^{j,1}_p\\   \hline
                                                                                                                                      c^{j,2}_1 \\
                                                                                                                                      \vdots \\
                                                                                                                                      c^{j,2}_p\\
                                                                                                                                    \end{array}
                                                                                                                                  \right)$ and one of the following may occur:
\begin{itemize}
\item either ${\bf c}^{j,1}={\bf 1}_p$ and ${\bf c}^{j,2}={\bf 0}_p$,
\item or ${\bf c}^{j,1}={\bf 0}_p$ and ${\bf c}^{j,2}={\bf 1}_p$,
\item or ${\bf c}^{j,1}$ and ${\bf c}^{j,2}$ have the same number of nonzero entries (each equal to $1$).
\end{itemize}
\end{enumerate}

\begin{remark}\rm
Observe that the conditions (2) and (3) can be interpreted as a sort of coupled equitability condition between the vertex subsets $V_1$ and $V_2$.
\end{remark}

\begin{example}\label{example3}\rm
Consider the graph $G$ on the left of Fig. \ref{exampleggmnonsegnato}, where $V_1=\{1,2,3\}$, $V_2=\{4,5,6\}$, and $V=\{7,8\}$. Here, we have $\ell=-1$. In the switching from $G$ to $G'$, the three edges connecting the vertex $7$ to $V_1$ are deleted and replaced by three edges connecting $7$ to $V_2$, whereas the two edges from $8$ to $V_1$ and the two edges from $8$ to $V_2$ remain unchanged. The adjacency matrix of $G$ is given by

$$
A_G= \left(
     \begin{array}{ccc|ccc|cc}
       0 & 0 & 0 & 1 & 0 & 0 & 1 & 0 \\
       0 & 0 & 1 & 1 & 0 & 1 & 1 & 1 \\
       0 & 1 & 0 & 0 & 1 & 1 & 1 & 1 \\   \hline
       1 & 1 & 0 & 0 & 0 & 1 & 0 & 0 \\
       0 & 0 & 1 & 0 & 0 & 0 & 0 & 1 \\
       0 & 1 & 1 & 1 & 0 & 0 & 0 & 1 \\   \hline
       1 & 1 & 1 & 0 & 0 & 0 & 0 & 0 \\
       0 & 1 & 1 & 0 & 1 & 1 & 0 & 0 \\
     \end{array}
   \right).
$$

The graph $G'$ is depicted on the right in Fig. \ref{exampleggmnonsegnato}. Notice that $G$ and $G'$ are cospectral, with

$$
p_G(x) = p_{G'}(x)=(x+1)^2(x^6-2x^5-11x^4+10x^3+22x^2-18x+1);
$$
however, $G$ and $G'$ are nonisomorphic, since the vertex $1$ is pendant in $G'$, but there is no pendant vertex in $G$.

\begin{figure}[h]
\begin{center}
\psfrag{1}{$1$}\psfrag{2}{$2$} \psfrag{3}{$3$} \psfrag{4}{$4$} \psfrag{5}{$5$} \psfrag{6}{$6$} \psfrag{7}{$7$} \psfrag{8}{$8$}
\psfrag{G}{$G$} \psfrag{G'}{$G'$}
\includegraphics[width=0.5\textwidth]{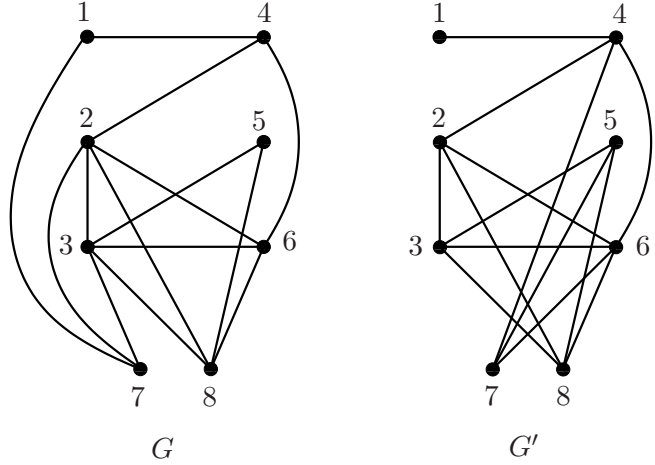} \caption{The graphs $G$ and $G'$ of Example \ref{example3}.}  \label{exampleggmnonsegnato}
\end{center}
\end{figure}
\end{example}

\section{The G-GM-switching for signed graphs}\label{sect5}
In the paper \cite{wangqiuhu} the condition that $p$ is an odd prime is used in order to prove some results concerning the possibility of a graph of admitting a G-GM-switching. However, the matrix machinery still works whenever the size of the parts $V_1$ and $V_2$ is any positive integer $m$. In this section, we extend the G-GM-switching construction to the more general setting of signed graphs, and we will obtain the unsigned G-GM-switching as a particular case. In our computations, the equal cardinality of the vertex subsets $V_1$ and $V_2$ is any positive integer $m$.

We define the \emph{signed G-GM-switching} for a given signed graph $\Gamma=(G,\sigma)$ as follows. Recall, let $\Gamma=(G,\sigma)$ be a signed graph, with underlying graph $G$, vertex set $V_G$ and edge set $E_G$, with signature function $\sigma:E_G\to \{\pm 1\}$. As usual, for a vertex $v\in V_G$, let us denote by $d^+(v)$ the number of positive edges incident to $v$ in $\Gamma$, and by $d^-(v)$ the number of negative edges incident to $v$ in $\Gamma$. Moreover, we put $d^{\pm}(v) = d^+(v) -d^-(v)$, which is called the \emph{net-degree} of $v$ in $\Gamma$.

Now assume that $V_1\sqcup V_2 \sqcup V$ is a partition of $V_G$, with $|V_1|=|V_2|=m$ and $|V|=d$. For a given vertex $v\in V_G$ let $d^+_i(v)$ denote the number of positive edges connecting $v$ to some vertex of $V_i$, and similarly for $d^-_i(v)$, so that the $i$-th net-degree of $v$ is defined as $d^{\pm}_i(v) = d^+_i(v)-d^-_i(v)$. Now suppose that:
\begin{enumerate}
\item $d^{\pm}_1(v)- d^{\pm}_2(v)= d^{\pm}_2(u) - d^{\pm}_1(u)=\ell$ for every $v\in V_1$ and $u\in V_2$;
\item for each $v\in V$:
\begin{itemize}
\item either $d^+_1(v)=m$ and $d^-_1(v) = d^+_2(v)=d^-_2(v)=0$ (and similarly by exchanging the role of $V_1$ and $V_2$);
\item or $d^-_1(v)=m$ and $d^+_1(v) = d^+_2(v)=d^-_2(v)=0$ (and similarly by exchanging the role of $V_1$ and $V_2$);
\item or $d^{\pm}_1(v)= d^{\pm}_2(v)$;
\item or $d^+_1(v)=d^-_2(v)=m$ (and similarly by exchanging the role of $V_1$ and $V_2$).
\end{itemize}
\end{enumerate}
Then the graph $\Gamma'$ constructed from $\Gamma$ by switching is obtained as follows: for every $v\in V$
\begin{itemize}
\item if $d^+_1(v)=m$ and $d^-_1(v) = d^+_2(v)=d^-_2(v)=0$, then one deletes the existing (positive) edges from $v$ to $V_1$ and join instead $v$ to all the vertices of $V_2$ by positive edges (and similarly by exchanging the role of $V_1$ and $V_2$);
\item if $d^-_1(v)=m$ and $d^+_1(v) = d^+_2(v)=d^-_2(v)=0$, then one deletes the existing (negative) edges from $v$ to $V_1$ and join instead $v$ to all the vertices of $V_2$ by negative edges (and similarly by exchanging the role of $V_1$ and $V_2$);
\item if $d^+_1(v)=d^-_2(v)=m$ (and similarly by exchanging the role of $V_1$ and $V_2$), then one changes the sign to all the (positive) edges joining $v$ to $V_1$ and to all the (negative) edges joining $v$ to $V_2$.
\end{itemize}

\begin{theorem}\label{wangsignedthm}
The graphs $\Gamma$ and $\Gamma'$ are cospectral.
\end{theorem}
\begin{proof}
The proof is based on the fact that the adjacency matrices $A_\Gamma$ and $A_{\Gamma'}$ are conjugated by the rational orthogonal matrix $Q = {\rm diag}(U_{2m}, I_d)$, where $U_{2m} = I_{2m}+\frac{1}{m}\left(
                                                                                                                                                                  \begin{array}{cc}
                                                                                                                                                                    -J_m & J_m \\
                                                                                                                                                                    J_m & -J_m \\
                                                                                                                                                                  \end{array}
                                                                                                                                                                \right)$
and
$$
A_\Gamma = \left(
           \begin{array}{cc}
             M & C \\
             C^T & D \\
           \end{array}
         \right),  \qquad \textrm{with } \qquad  M=\left(
           \begin{array}{cc}
             A_1 & B \\
             B^T & A_2 \\
           \end{array}
         \right).
$$
Observe that the hypothesis on the graph $\Gamma$ implies that:
\begin{enumerate}
\item $A_1$ and $A_2$ are symmetric matrices of order $m$;
\item for each $i,i'=1,\ldots ,m$, one has
\begin{eqnarray}\label{conditionAB}
\sum_{j=1}^m ((A_1)_{ij}-B_{ij}) =  \sum_{j=1}^m ((A_2)_{i'j}-B_{ji'})= \ell;
\end{eqnarray}

\item each column ${\bf c}^j$ of $C$ is a vector of length $2m$ which can be partitioned into two semicolumns ${\bf c}^{j,1}$ and ${\bf c}^{j,2}$ of length $m$ as ${\bf c}^j = \left(
                                                                                                                                    \begin{array}{c}
                                                                                                                                      c^{j,1}_1 \\
                                                                                                                                      \vdots \\
                                                                                                                                      c^{j,1}_m\\   \hline
                                                                                                                                      c^{j,2}_1 \\
                                                                                                                                      \vdots \\
                                                                                                                                      c^{j,2}_m\\
                                                                                                                                    \end{array}
                                                                                                                                  \right)$ and one of the following may occur:
\begin{itemize}
\item either ${\bf c}^{j,1}={\bf 1}_m$ and ${\bf c}^{j,2}={\bf 0}_m$ or viceversa;
\item or ${\bf c}^{j,1}=-{\bf 1}_m$ and ${\bf c}^{j,2}={\bf 0}_m$ or viceversa;
\item or $\sum_{h=1}^m c^{j,1}_h =  \sum_{h=1}^m c^{j,2}_h$;
\item or ${\bf c}^{j,1}={\bf 1}_m$ and ${\bf c}^{j,2}=-{\bf 1}_m$ or viceversa.
 \end{itemize}
\end{enumerate}
We have
$$
QA_\Gamma Q = \left(
        \begin{array}{cc}
          U_{2m}& 0 \\
          0 & I_d \\
        \end{array}
      \right)      \left(
        \begin{array}{cc}
          M & C \\
          C^T & D \\
        \end{array}
      \right)
 \left(
        \begin{array}{cc}
          U_{2m} & 0 \\
          0 & I_d \\
        \end{array}
      \right)    = \left(
        \begin{array}{cc}
          U_{2m}MU_{2m} & U_{2m}C \\
          C^TU_{2m} & D \\
        \end{array}
      \right).
$$
We start by proving that $U_{2m}MU_{2m} =M$. A direct computation gives $U_{2m}MU_{2m} = M + N$, with
\begin{eqnarray*}
N&=& \frac{1}{m}\left(
        \begin{array}{cc}
          A_1 & B \\
          B^T & A_2 \\
        \end{array}
      \right)
 \left(
        \begin{array}{cc}
          -J_m & J_m \\
          J_m & -J_m \\
        \end{array}
      \right) + \frac{1}{m}
 \left(
        \begin{array}{cc}
          -J_m & J_m \\
          J_m & -J_m \\
        \end{array}
      \right)
 \left(
        \begin{array}{cc}
          A_1 & B \\
          B^T & A_2 \\
        \end{array}
      \right)\\      &+& \frac{1}{m^2}
 \left(
        \begin{array}{cc}
          -J_m & J_m \\
          J_m & -J_m \\
        \end{array}
      \right)
 \left(
        \begin{array}{cc}
          A_1 & B \\
          B^T & A_2 \\
        \end{array}
      \right)
 \left(
        \begin{array}{cc}
          -J_m & J_m \\
          J_m & -J_m \\
        \end{array}
      \right).
\end{eqnarray*}
Therefore, our goal is to prove that $N= \left(
        \begin{array}{cc}
          N_{11} & N_{12} \\
          N_{12}^T & N_{22} \\
        \end{array}
      \right)
$ is the zero matrix. By using the property in Eq. \eqref{conditionAB} and the definition of $J_m$, we get:
\begin{eqnarray*}
N_{11} &=& \frac{1}{m}(-A_1J_m+BJ_m-J_mA_1+J_mB^T) + \frac{1}{m^2}(J_mA_1J_m-J_mB^TJ_m-J_mBJ_m+J_mA_2J_m)\\
&=& \frac{1}{m}(-(A_1-B)J_m-J_m(A_1-B^T)) + \frac{1}{m^2}(J_m(A_1-B)J_m + J_m(A_2-B^T)J_m)\\
&=& \frac{1}{m}(-\ell J_m -\ell J_m) + \frac{1}{m^2}(J_m\ell J_m + J_m\ell J_m)\\
&=& -\frac{2\ell}{m}J_m +\frac{2\ell}{m^2}mJ_m =O_m.
\end{eqnarray*}
A similar argument shows that $N_{22}=O_m$. Now:
\begin{eqnarray*}
N_{12} &=& \frac{1}{m}(A_1J_m-BJ_m-J_mB+J_mA_2) + \frac{1}{m^2}(-J_mA_1J_m+J_mB^TJ_m+J_mBJ_m-J_mA_2J_m)\\
&=& \frac{1}{m}((A_1-B)J_m+J_m(A_2-B)) + \frac{1}{m^2}(J_m(-A_1+B)J_m + J_m(-A_2+B^T)J_m)\\
&=& \frac{1}{m}(\ell J_m +\ell J_m) + \frac{1}{m^2}(-J_m\ell J_m - J_m\ell J_m)\\
&=& \frac{2\ell}{m}J_m -\frac{2\ell}{m^2}mJ_m =O_m.
\end{eqnarray*}
In order to conclude the proof, we have to show that the left multiplication of $C$ by $U_{2m}$ changes the column  ${\bf c}^j = \left(
                                                                                                                                    \begin{array}{c}
                                                                                                                                      {\bf c}^{j,1}\\ \hline
                                                                                                                                      {\bf c}^{j,2}                                                                                                   \end{array}
                                                                                                                                  \right)$  according to the switching rules of the construction of $\Gamma'$.
Observe that
$$
U_{2m}\left(
                                                                                                                                    \begin{array}{c}
                                                                                                                                      {\bf c}^{j,1}\\
                                                                                                                                      {\bf c}^{j,2}                                                                                                   \end{array}
                                                                                                                                  \right) = \left(
                           \begin{array}{cc}
                             I_m-\frac{1}{m}J_m & \frac{1}{m}J_m \\
                             \frac{1}{m}J_m  & I_m-\frac{1}{m}J_m \\
                           \end{array}
                         \right)\left(
                                                                                                                                    \begin{array}{c}
                                                                                                                                      {\bf c}^{j,1}\\
{\bf c}^{j,2}                                                                                                   \end{array}
                                                                                                                                  \right) =
\left(
  \begin{array}{c}
   (I_m-\frac{1}{m}J_m){\bf c}^{j,1} + \frac{1}{m}J_m{\bf c}^{j,2}  \\
    \frac{1}{m}J_m{\bf c}^{j,1} + (I_m-\frac{1}{m}J_m){\bf c}^{j,2} \\
  \end{array}
\right)
$$
so that, for each $i=1,\ldots, m$:
$$
(U_{2m}C)^{j,1}_i = c^{j,1}_i + \frac{1}{m}\sum_{h=1}^m(c^{j,2}_h-c^{j,1}_h)
$$
$$
(U_{2m}C)^{j,2}_i = c^{j,2}_i + \frac{1}{m}\sum_{h=1}^m(c^{j,1}_h-c^{j,2}_h).
$$
Then it can be easily checked that:
\begin{itemize}
\item if ${\bf c}^{j,1}={\bf 1}_m$ and ${\bf c}^{j,2}={\bf 0}_m$, then $(U_{2m}C)^{j,1} = {\bf 0}_m$ and  $(U_{2m}C)^{j,2}={\bf 1}_m$ (and viceversa);
\item if ${\bf c}^{j,1}=-{\bf 1}_m$ and ${\bf c}^{j,2}={\bf 0}_m$, then $(U_{2m}C)^{j,1} = {\bf 0}_m$ and  $(U_{2m}C)^{j,2}=-{\bf 1}_m$ (and viceversa);
\item if $\sum_{h=1}^m c^{j,1}_h =  \sum_{h=1}^m c^{j,2}_h$, then $(U_{2m}C)^{j,1} = {\bf c}^{j,1}$ and  $(U_{2m}C)^{j,2}={\bf c}^{j,2}$;
\item if ${\bf c}^{j,1}={\bf 1}_m$ and ${\bf c}^{j,2}=-{\bf 1}_m$, then $(U_{2m}C)^{j,1} = -{\bf 1}_m$ and $(U_{2m}C)^{j,2}={\bf 1}_m$ (and viceversa).
\end{itemize}
\end{proof}

\begin{example}\label{exampleggmsegnatoes}\rm
Consider the graph $\Gamma$ represented on the left in Fig. \ref{exampleggmsegnato}, where $V_1=\{1,2,3,4,5\}$, $V_2=\{6,7,8,9,10\}$, and $V=\{11,12,13,14\}$.
\begin{figure}[h]
\begin{center}
\psfrag{1}{$1$}\psfrag{2}{$2$} \psfrag{3}{$3$} \psfrag{4}{$4$} \psfrag{5}{$5$} \psfrag{6}{$6$} \psfrag{7}{$7$} \psfrag{8}{$8$}
\psfrag{9}{$9$}\psfrag{10}{$10$} \psfrag{11}{$11$} \psfrag{12}{$12$} \psfrag{13}{$13$} \psfrag{14}{$14$}
\psfrag{Gamma}{$\Gamma$}\psfrag{Gamma'}{$\Gamma'$}
\includegraphics[width=1\textwidth]{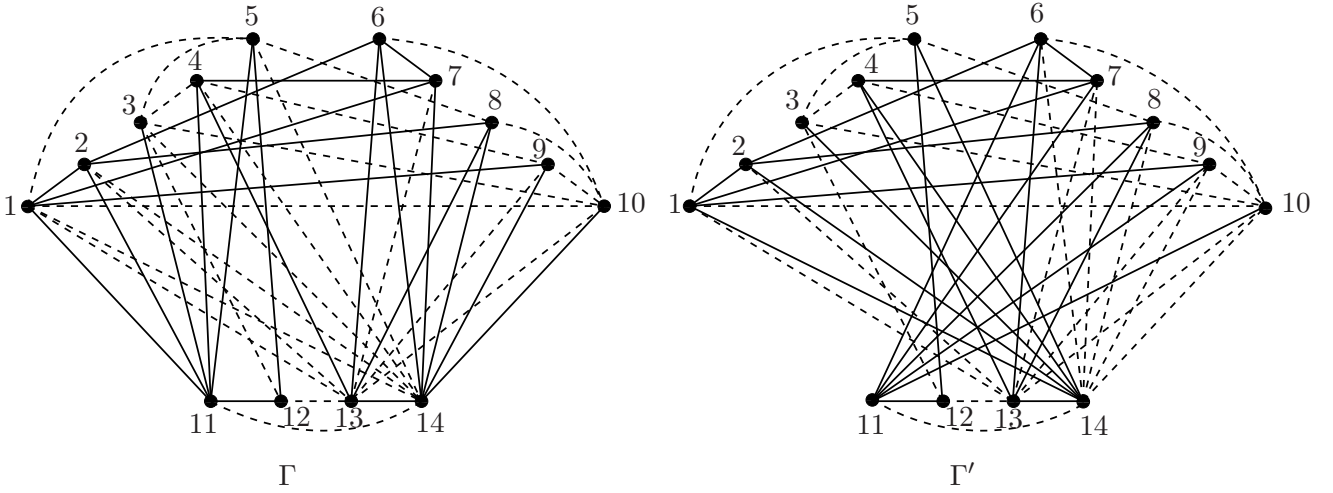} \caption{The graphs $\Gamma$ and $\Gamma'$ of Example \ref{exampleggmsegnatoes}.}  \label{exampleggmsegnato}
\end{center}
\end{figure}
The adjacency matrix of $\Gamma$ is
$$
A_\Gamma = \left(
      \begin{array}{ccccc|ccccc|cccc}
        0 & 1 & 0 & 0 & -1 & 0 & 1 & 0 & 1 & -1 & 1 & 0 & -1 & -1 \\
        1 & 0 & 0 & 0 & 0 & 1 & 0 & 1 & 0 & 0 & 1 & 0 & -1 & -1 \\
        0 & 0 & 0 & -1 & -1 & 0 & 0 & 0 & 0 & -1 & 1 & -1 & 0 & -1 \\
        0 & 0 & -1 & 0 & 0 & 0 & 1 & 0 & -1 & 0 & 1 & 0 & 1 & -1 \\
        -1 & 0 & -1 & 0 & 0 & 0 & 0 & -1 & 0 & 0 & 1 & 1 & 0 & -1 \\  \hline
        0 & 1 & 0 & 0 & 0 & 0 & 1 & 0 & 0 & -1 & 0 & 0 & 1 & 1 \\
        1 & 0 & 0 & 1 & 0 & 1 & 0 & 0 & 0 & 0 & 0 & 0 & -1 & 1 \\
        0 & 1 & 0 & 0 & -1 & 0 & 0 & 0 & 0 & -1 & 0 & 0 & 1 & 1 \\
        1 & 0 & 0 & -1 & 0 & 0 & 0 & 0 & 0 & -1 & 0 & 0 & -1 & 1 \\
        -1 & 0 & -1 & 0 & 0 & -1 & 0 & -1 & -1 & 0 & 0 & 0 & -1 & 1 \\   \hline
        1 & 1 & 1 & 1 & 1 & 0 & 0 & 0 & 0 & 0 & 0 & 1 & 0 & -1 \\
        0 & 0 & -1 & 0 & 1 & 0 & 0 & 0 & 0 & 0 & 1 & 0 & -1 & 0 \\
        -1 & -1 & 0 & 1 & 0 & 1 & -1 & 1 & -1 & -1 & 0 & -1 & 0 & 1 \\
        -1 & -1 & -1 & -1 & -1 & 1 & 1 & 1 & 1 & 1 & -1 & 0 & 1 & 0 \\
      \end{array}
    \right)
$$
which satisfies the conditions of Theorem \ref{wangsignedthm}, with $\ell=-1$. The graph $\Gamma'$ is depicted on the right in Fig. \ref{exampleggmsegnato}, and it is obtained from $\Gamma$ by deleting all the positive edges from the vertex $11$ to $V_1$, and introducing positive edges connecting that vertex to each vertex of $V_2$; by changing the sign of the five positive edges connecting the vertex $14$ to $V_1$, and of the five negative edges connecting the vertex $14$ to $V_2$. The graphs $\Gamma=(G,\sigma)$ and $\Gamma'=(G', \sigma')$ are cospectral, with
\begin{eqnarray*}
p_{\Gamma}(x)=p_{\Gamma'}(x) &=& x^{14}-46x^{12}+770x^{10}-22x^9-5754x^8+318x^7+18879x^6-608x^5\\
&-&23953x^4+422x^3+9520x^2-482x-489.
\end{eqnarray*}
However, the underlying graphs $G$ and $G'$ are not cospectral, since
\begin{eqnarray*}
p_G(x) &=& x^{14}-46x^{12}-100x^{11}+354x^{10}+1078x^9-610x^8-3230x^7+59x^6\\
&+&3636x^5-53x^4-1430x^3+324x^2-6x-1
\end{eqnarray*}
\begin{eqnarray*}
p_{G'}(x) &=& x^{14}-46x^{12}-96x^{11}+334x^{10}+826x^9-978x^8-2378x^7+1539x^6\\
&+&2592x^5-1293x^4-798x^3+368x^2-26x-1.
\end{eqnarray*}
Therefore, $\Gamma$ and $\Gamma'$ are not switching isomorphic as signed graphs.
\end{example}

\begin{remark}\rm
The condition $d^{\pm}_1(v)=d^{\pm}_2(v)$ for a vertex $v\in V$ contains the following special subcases:
\begin{itemize}
\item $v$ has no neighbors either in $V_1$ or in $V_2$;
\item $d^+_1(v)=d^-_1(v)\neq 0$ and $d^+_2(v)=d^-_2(v) =0$ (and similarly by exchanging the role of $V_1$ and $V_2$);
\item $d^+_1(v) = d^+_2(v)\neq 0$ and $d^-_1(v)=d^-_2(v)=0$;
\item $d^-_1(v) = d^-_2(v)\neq 0$ and $d^+_1(v)=d^+_2(v)=0$.
\end{itemize}
Moreover, our Theorem \ref{wangsignedthm} contains and extends the theorem by Wang, Qiu and Hu (Theorem \ref{wangqiuhu}), as in the unsigned case the signature is constantly equal to $1$, and the condition on the constant difference $d^{\pm}_1 (v) - d^{\pm}_2 (v) = d^{\pm}_2 (u) - d^{\pm}_1(u) =\ell$ reduces to the condition $d^{+}_1 (v) - d^{+}_2 (v) = d^{+}_2 (u) - d^{+}_1(u) =\ell$, for each $v\in V_1$ and $u\in V_2$. On the other hand, the condition $d^{\pm}_1(v)=d^{\pm}_2(v)$ for each $v\in V$ is equivalent to the condition $d^{+}_1(v)=d^{+}_2(v)$, which appears in the unsigned theorem.
\end{remark}

\section*{Acknowledgments}

The first and second authors recognize the support from the INDAM-GNSAGA.

%%%%%%%%%%%%%%%%%%%%%%%%%%%%%%%%%%%%%%%%%%%%%%%%%%%%%%%%%%%%%%%%%%%%%%%%%%%%%%%%%%%%

\end{document}